\documentclass{amsart}

\usepackage{etex}
\usepackage{amssymb, amsthm, amsmath, pifont,enumerate} 
\usepackage[usenames]{color}
\usepackage[english]{babel} 
\usepackage[noadjust]{cite}
\usepackage[all]{xy}
\usepackage{amsfonts}
\usepackage{graphicx}
\usepackage{verbatim}
\usepackage{pstricks}
\usepackage[utf8x]{inputenc}
\usepackage{mathrsfs}
\usepackage{pst-node}
\usepackage{tikz}
\def\N{\mathbb{N}}

\def\C{\mathbb{C}}
\def\P{\mathbb{P}}

\newtheorem{theorem}{Theorem}[section]
\newtheorem*{theorem*}{Theorem}
\newtheorem{corollary}[theorem]{Corollary}
\newtheorem*{corollary*}{Corollary}
\newtheorem*{maintheorem*}{Main Theorem}
\newtheorem*{fact*}{Fact}
\newtheorem{proposition}[theorem]{Proposition}
\newtheorem{lemma}[theorem]{Lemma}
\newtheorem*{lemma*}{Lemma}

\newtheorem*{observation*}{Theorem of Zariski}

\theoremstyle{definition}

\theoremstyle{definition}
\newtheorem{remark}[theorem]{Remark}
\theoremstyle{definition}
\newtheorem{definition}[theorem]{Definition}

\title{Complete algebraic vector fields on Danielewski surfaces}
\author{Matthias Leuenberger}
\address{Institute of Mathematics\\ University of Bern\\ Sidlerstrasse 5\\ CH-3012 Bern\\ Switzerland}
\email{matthias.leuenberger@bluewin.ch}

\begin{document}
\thanks{The author is supported by SNF Grant 200021-140235/1}
\begin{abstract}
We give the classification of all complete algebraic vector fields on Danielewski surfaces (smooth surfaces given by $xy=p(z)$). We use the fact that for
each such vector field
there exists a certain fibration that is preserved under its flow. In order to get the explicit list of vector fields a classification of regular function
with general fiber $\C$ or $\C^*$ is required. In this text we present results about such fibrations on Gizatullin surfaces and we give a precise description
of these fibrations for Danielewski surfaces.
\end{abstract}

\maketitle

\section{Introduction}
Complete (= globally integrable) vector fields are vector fields for which a global holomorphic flow map exists. In general the problem of classifying
complete vector fields on Stein manifolds seems to be out of reach. However, for complete algebraic vector fields on affine varieties there are some known
results. In
2000 Anders\'en \cite{andersen} gave a classification of complete algebraic vector fields on $(\C^*)^n$. For affine surfaces the situation looks better.
In 2004 Brunella \cite{Brunella} gave an explicit
classification of complete algebraic vector fields on $\C^2$. The proof uses deep results from the theory of foliations on projective surfaces
developed by Brunella \cite{BrunellaFoliations,BrunellaFoliationsSurvey}, McQuillan \cite{McQuillan}, and others. From this theory it follows that there is 
always a regular function with general fibers isomorphic to $\C$ or $\C^*$ such that the
vector field sends fibers to fibers. Since these functions on $\C^2$ where classified by Suzuki \cite{Suzuki} it was only a small step to conclude the explicit 
form of
the
complete algebraic vector fields on $\C^2$. An extension of this result to affine toric surfaces (a quotient of $\C^2$ by some cyclic group action) has been
recently presented in \cite{kll}. The fact that each complete algebraic vector field preserves the fibers of a regular function with
$\C$ or $\C^*$ fibers turns out to be true on almost all normal affine surfaces. This makes it possible to classify all complete algebraic vector fields for 
other surfaces. 
\begin{fact*}[{\cite[Theorem 1.3]{wfhs}}]
 Let $S$ be a normal affine surface such that not all complete algebraic vector fields on $S$ are proportional, and let $\nu$ be a complete algebraic vector 
field on $S$. Then there exists a regular function $f: S\rightarrow \C$ with
general fiber isomorphic to $\C$ or $\C^*$ such that the flow of $\nu$ sends fibers of $f$ to fibers of $f$ (in short: $\nu$ preserves the fibration $f$).
\end{fact*}
This fact shows that once the classification of $\C$- and $\C^*$-fibrations is done the complete vector fields are described. In this text we give some
results
about $\C$- and $\C^*$-fibrations on Gizatullin surfaces. For the special case of smooth surfaces given by $xy=p(z)$ (which are called Danielewski surfaces) we
can provide a precise classification: Here we give the list of complete algebraic vector fields. Since the $\C$- and $\C^*$-polynomials on Danielewski surfaces
look much alike the ones on $\C^2$ the vector fields also look similarly. Surprisingly if $\deg(p)=4$ there occurs a complete vector field that has no analogue
on $\C^2$.

Section 2 is a recapitulation of the definition of Gizatullin surfaces, a generalization of Danielewski surfaces, and SNC-completions, a
powerful tool for affine algebraic surfaces. In section 3 we present some results about $\C$- and $\C^*$-fibrations on Gizatullin surfaces which will be used in
section 4 to give an explicit description of $\C$- and $\C^*$-fibrations on Danielewski surfaces. Section 5 combines this description to a proof of the
following theorem:

\begin{maintheorem*}
 Let $\nu$ be a complete algebraic vector field on $S=\lbrace xy=p(z)\rbrace$ (where $p$ has simple zeros), and let the hyperbolic vector field (HF) and the two
shear vector fields (SF) be defined as follows: 
$$\mathrm{HF}
= x\frac{\partial}{\partial x} - y\frac{\partial}{\partial y}, \quad\mathrm{SF}^x = p'(z)\frac{\partial}{\partial y} +
x\frac{\partial}{\partial z}\quad  \mathrm{and} \quad \mathrm{SF}^y = p'(z)\frac{\partial}{\partial x} +
y\frac{\partial}{\partial z}. $$Then $\nu$ occurs in the following list  (up to an automorphism of $S$):

(1) $\nu$ preserves the polynomial $x$ and is of the form:
\[\nu = c\mathrm{HF} + \left( A(x)z + B(x) \right)\mathrm{SF}^x\]
for some $c\in\C$ and $A,B\in\C[x]$.

(2) $\nu$ preserves a polynomial $x^m(x^l(z+a)+Q(x))^n$ for coprime numbers $m,n\in\N$, $\deg(Q)\in\N_0$, $a\in\C$ and $\deg(Q) < l$ and is of the form:
\begin{eqnarray*} \nu &=& c\left( \frac{z+a}{x} + \frac{Q(x)}{x^{l+1}}\right)\mathrm{SF}^x +
A(x^m(x^l(z+a)+Q(x))^n) \\ &&\cdot\left[n\mathrm{HF}-\left(\frac{(m+nl)(z+a)}{x} + \frac{mQ(x)+nxQ'(x)}{x^{l+1}}\right)\mathrm{SF}^x\right]\end{eqnarray*}
for some $c\in\C$ and $A\in\C[t]$ satisfying $A(0)=c/(m+nl)$ and $A(x^m(x^l(z+a)+Q(x))^n)(mQ(x)+nxQ'(x))-cQ(x) \in x^{l+1}\cdot\C[S]$.

(3) If $\deg(p)=4$ then $\nu$ can also preserve the polynomial $ ax + y + \frac{1}{6}p''(z)$ where $a$ is the leading coefficient of $p$. In this case $\nu$
looks
like:
\[
 \nu = A\left(ax + y +
\frac{1}{6}p''(z)\right)\left(-\frac{1}{6}p'''(z)\mathrm{HF} + a\mathrm{SF}^x -\mathrm{SF}^y\right)\]
for some $A\in\C[t]$.
\end{maintheorem*}

The Main Theorem describes a class of one-parameter subgroups of the group of holomorphic automorphisms on $S=\lbrace xy=p(z)\rbrace$. It is worth to compare 
this result to well known results in the algebraic case. Daigle \cite{Daigle} and Makar-Limanov \cite{ML} showed that on $S$ every 
algebraic $\C^+$-action is up to an algebraic automorphisms induced by some vector field $f(x)\mathrm{SF}^x$ for some polynomial $f\in\C[x]$, and thus is a 
special case of (1) of the Main Theorem. Moreover, by \cite{FKZ} there is a unique (up to an automorphism) algebraic $\C^*$-action on $S$ which is induced by 
$\mathrm{HF}$, which can be seen as a vector field of type (1) or (2). 

\textbf{Acknowledgements:} I thank Shulim Kaliman for introducing me into this interesting topic and for his helpful comments on this article. Additionally, I 
thank the referee for the carefully done report and the numerous style remarks.

\section{Gizatullin surfaces and their completions}
\subsection{SNC-completions and dual graphs}
It is a well established procedure in affine algebraic geometry to use so called SNC-completions of affine surfaces. 
Let $S$ be an affine surface, and let $X
\supset S$ be a projective surface such that the \textit{boundary divisor} $D=X\setminus S=C_1\cup\ldots\cup C_k$ is contained in the
smooth locus of $X$. If moreover the curves $C_i$ are smooth and intersect pairwise transversally and at most in double points of $D$ then we say that $X$ is a
completion of $S$ with \textit{simple normal crossings} (in short: \textit{SNC-completion}). Every normal affine surface admits an SNC-completion. In this text 
$D=X\setminus S$ 
will
always be a union of rational curves.

A good reference for SNC-completions is for example \cite{FKZCompletions}. In particular, in this reference most notions that are used in this section are 
introduced. However the concept of SNC-completions was already used much earlier by Danilov and Gizatullin. Let $X$ be an SNC-completion of an affine surface 
$S$ then its \textit{dual graph} $\Gamma_X$ is given as follows: The vertices of $\Gamma_X$ are given by the
irreducible components $C_i$ of the boundary $D=X\setminus S$ and each intersection point $p\in C_i\cap C_j$ of two different components corresponds to an edge
of $\Gamma_X$ that connects the vertices which correspond to $C_i$ and $C_j$. The graph $\Gamma_X$ is often considered as a weighted graph where the weight of
a vertex is given by the self-intersection $C_i\cdot C_i$ of its corresponding curve $C_i$.

Clearly neither SNC-completions nor dual graphs are unique: Modifications along the boundary will change the boundary and the dual graph of the boundary. The
following two modifications (and its inverses) are possible ($C$ is the name of the vertex and $\omega=C\cdot C$ is its weight):

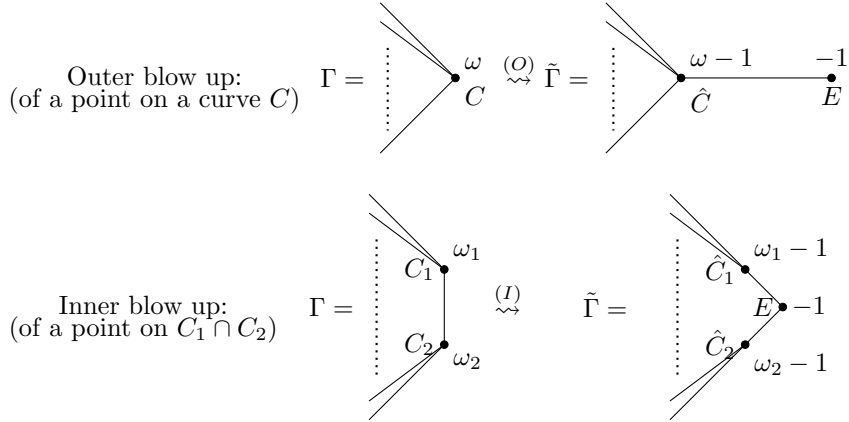
\vspace{0.5 cm}
\begin{tikzpicture}
 \node {Outer blow up:} node[below] {(of a point on a curve $C$)};
\node at (2.5,0) {$\Gamma=$};
 \draw (3,1)--(4,0)--(3,-1);
 \draw (3,0.75)--(4,0);
\draw[fill] (4,0) circle (0.05) node[above right] {$\omega$} node[below right] at (4,0) {$C$};
\draw[dotted, thick] (3.1,0.4)--(3.1,-0.7);
\node at (5.2,0.1) {$\stackrel{(O)}{\rightsquigarrow}$ $\tilde\Gamma=$};
\draw (6,1)--(7,0)--(6,-1);
 \draw (6,0.75)--(7,0);
\draw[fill] (7,0) circle (0.05) node[above right] {$\omega-1$} node[below right] at (7,0) {$\hat C$};
\draw[dotted, thick] (6.1,0.4)--(6.1,-0.7);
\draw (7,0)--(9,0);
\draw[fill] (9,0) circle (0.05) node[above] {$-1$} node [below] at (9,0) {$E$};
\end{tikzpicture}

\vspace{0.5 cm}
\begin{tikzpicture}
 \node {Inner blow up:} node[below] {(of a point on $C_1\cap C_2$)};
\node at (2.5,0) {$\Gamma=$};
 \draw (3,1.5)--(4,0.5)--(4,-0.5)--(3,-1.5);
 \draw (3,1.25)--(4,0.5);
 \draw (3,-1.25)--(4,-0.5);
\draw[fill] (4,0.5) circle (0.05) node[left] {$C_1$} node[above right] at (4,0.5) {$\omega_1$};
\draw[fill] (4,-0.5) circle (0.05) node[left] {$C_2$} node[below right] at (4,-0.5) {$\omega_2$};
\draw[dotted, thick] (3.1,0.9)--(3.1,-0.9);
\node at (5.2,0.1) {$\quad\quad\stackrel{(I)}{\rightsquigarrow}$ $\quad\quad \tilde\Gamma=$};
 \draw (7,1.5)--(8,0.5)--(8.5,0)--(8,-0.5)--(7,-1.5);
 \draw (7,1.25)--(8,0.5);
 \draw (7,-1.25)--(8,-0.5);
\draw[fill] (8,0.5) circle (0.05) node[left] {$\hat C_1$} node[above right] at (8,0.5) {$\omega_1-1$};
\draw[fill] (8,-0.5) circle (0.05) node[left] {$\hat C_2$} node[below right] at (8,-0.5) {$\omega_2-1$};
\draw[fill] (8.5,0) circle (0.05) node[left] {$E$} node[right] at (8.5,0) {$-1$};
\draw[dotted, thick] (7.1,0.9)--(7.1,-0.9);
\end{tikzpicture}
\vspace{0.5 cm}
\newline where $E$ is the exceptional divisor and $\hat C$ denotes the strict transform of a curve $C$. A sequence of $(I)$, $(I^{-1})$, $(O)$ and $(O^{-1})$
starting
with a weighted graph is called a modification of weighted graphs.
A birational map $\varphi: X \dashrightarrow Y$ between two completions $X,Y$ of an affine surface $S$ such that $\varphi\vert_S$ induces an
isomorphism on $S$ is called a birational modification of completions and an isomorphism of completions if $\varphi$ is additionally an isomorphism. By
a classical theorem of Zariski any birational map can be seen as composition of blow ups followed by a composition of blow downs. Hence we get the following
statement:
\begin{observation*}
(1) Let $S$ be an affine surface and let $X$ and $Y$ be two SNC-completions. Then there exists a SNC-completion $Z$ of $S$ obtained via a
sequence of blow ups performed over the boundaries of $S$ in $X$ and $Y$, respectively. Hence $\Gamma_Z$ is obtained by modifications as above from both 
$\Gamma_X$ and $\Gamma_Y$.

(2) Let $\gamma: \Gamma_X \rightsquigarrow \Gamma$ be a modification of weighted graphs. Then there is a completion $Y$ of $S$ such that $\Gamma_Y=\Gamma$ and
$Y$ is obtained
from $X$ by a birational map $\phi: X\dashrightarrow  Y$ that induces the modification $\gamma$ on the dual graphs. If $\gamma$ does not contain outer blow ups
then $\phi$ is uniquely determined.
\end{observation*}

A completion $X$ will be called minimal if $\Gamma_X$ does not have a (-1)-vertex of degree $\leq 2$. 

\subsection{Gizatullin surfaces}
A \textit{Gizatullin surface} is a normal affine surface $S$ that admits an SNC-completion $X$ such that the graph $\Gamma_X$ is linear. For such a completion 
(also
called a \textit{zigzag}) with
\begin{center}
\begin{tikzpicture}
\node[left] at (0,0) {$\Gamma_X = \quad\quad$};
\draw (0,0)--(2,0);
\draw[dotted] (2,0)--(3,0);
\draw (3,0)--(4,0);
\draw[fill] (0,0) circle(0.05) node[above]{$\omega_0$} node[below]{$C_0$};
\draw[fill] (1,0) circle(0.05) node[above]{$\omega_1$} node[below]{$C_1$};
\draw[fill] (4,0) circle(0.05) node[above]{$\omega_k$} node[below]{$C_k$};
\end{tikzpicture}
\end{center}
we use the notation 
$$\Gamma_X= [[\omega_0,\omega_1, \ldots,\omega_k]].$$
A completion $X$ is called \textit{standard} if 
$$\Gamma_X=[[0,0,\omega_2,\ldots,\omega_k]] \quad\mathrm{or}\quad \Gamma_X=[[0,0,0]] \quad \mathrm{or}\quad \Gamma_X=[[0,0]]$$
and $\omega_1$-\textit{semistandard} if 
$$\Gamma_X=[[0,\omega_1,\omega_2,\ldots,\omega_k]] \quad\mathrm{or}\quad \Gamma_X=[[0,\omega_1,0]] \quad \mathrm{or}\quad \Gamma_X=[[0,\omega_1]]$$
with $\omega_i\leq -2$ for all $2\leq i \leq k$.

Now we introduce two modifications of the boundary of Gizatullin surfaces. The first one is
\begin{equation}\label{makezero}[[0,\omega_1,\ldots]] \stackrel{(O)}{\rightsquigarrow}
[[-1,-1,\omega_1,\ldots]] \stackrel{(I^{-1})}{\rightsquigarrow}
[[0,\omega_1 +1,\ldots]]\end{equation}
that allows to transform any semistandard completion $X$ with $\Gamma_X = [[0,\omega_1,\allowbreak \omega_2,\ldots]]$ into a standard completion $Y$ with
$\Gamma_Y = [[0,0,\omega_2,\ldots]]$. The
second modification is a way to use a zero vertex in order to move weight from one side of the vertex to the other:
\begin{equation}\label{movezeros}
[[\ldots,\omega_{i-1},0,\omega_{i+1}\ldots]]\stackrel{(I)}{\rightsquigarrow}
[[\ldots[\omega_{i-1},-1,-1,\omega_{i+1}-1,\ldots]]
\end{equation}
$$\stackrel{(I^{-1})}{\rightsquigarrow} [[\ldots,\omega_{i-1}+1,0,\omega_{i+1}-1,\ldots]]$$
By a sequence of modification of type (\ref{movezeros}) it is possible to move zeros vertices through the boundary divisor:
$$ [[0,0,\omega_2,\ldots,\omega_k]] \rightsquigarrow [[\omega_2,0,0,\ldots,\omega_k]] \rightsquigarrow \cdots \rightsquigarrow
[[\omega_2,\ldots,\omega_k,0,0]].$$
The modification above is called \textit{reversion} and it shows that the data of a standard completion is in general not unique. Using these modifications we 
see that each Gizatullin surface admits a standard completion and that all minimal completions have a linear dual graph:
\begin{proposition}[\cite{FKZCompletions}] \label{standardcompletion}
 Let $S$ be a Gizatullin surface. Then: 

(1) There exists a standard completion $X$, and the dual graph $\Gamma_X$ is unique up to reversion.

(2) For any completions $X$ there is a contraction (i.e. a modification consisting of ($O^{-1}$) and ($I^{-1}$)) of $\Gamma_X$ to a linear graph. 
\end{proposition}

\section{$\C$- and $\C^*$-fibrations on Gizatullin surfaces}
Let us start with some well know facts in algebraic geometry.
\begin{proposition}\label{fibrations}
\begin{enumerate}[(a)]
 \item \cite{BookComplexSurfaces} Let $S$ be a normal affine surface, and let $f: S \rightarrow \C$ be a reduced\footnote{Recall that a regular function $f: 
S\rightarrow \C$ is called reduced if its general fiber is connected.} regular function with rational fibers.
Then there is a pseudo-minimal
SNC-completion $X$ such that $f$ extends to a regular function $\bar f: X \rightarrow \P^1$ with general fibers isomorphic to $\P^1$.
\item \cite{Suzuki} Let $f$ be as in (a) then $\chi(S)=\chi(F\times \C) + \sum (\chi(F')- \chi(F))$ where $\chi$ denotes the Euler characteristic, $F$ is
a regular fiber of
$f$ and the sum is taken over all singular fibers $F'$.
\item \cite{BookComplexSurfaces} Let $X$ be a smooth projective surface and let $f: X \rightarrow \P^1$ a regular function with general fiber isomorphic to
$\P^1$. Then there is a
sequence of contractions $\pi: X \rightarrow Y$ and a map $f': Y \rightarrow \P^1$ such that $f=f'\circ\pi$ and $f'$ is a $\P^1$-bundle.
\item \cite{FKZ} Let $C\cong \P^1$ be a curve on a rational projective surface $X$ with $C\cdot C = 0$ then there is a regular function $f: X \rightarrow \P^1$
such that
$C = f^{-1}(\infty)$ is a regular fiber of $f$.
\end{enumerate}
\end{proposition}

\begin{definition}
If a regular function $f:S\rightarrow \C$ (or $\P^1$) on a variety $S$ is considered as \textit{fibration} it means that we are only interested in its level 
sets (i.e.
the fibers). In particular, two regular functions are considered to be the same fibrations whenever the differ only by a M\"obius transform in the target. A
fibration $f$ is said to be a $\C$- (resp. $\C^*$- or $\P^1$-) fibration if its regular fiber is isomorphic to $\C$ (resp. $\C^*$ or $\P^1$).
\end{definition}

Let $S$ be a Gizatullin surface, and let $X$ be a standard completion with boundary $D=X\setminus S= C_0 \cup \ldots \cup C_k$. Then the two curves $C_0$
and $C_1$ induce (Proposition \ref{fibrations}(d)) both a regular function $\phi_0,\phi_1: X \rightarrow \P^1$ such that $C_0 = \phi_0^{-1}(\infty)$ and $C_1 =
\phi_1^{-1}(\infty)$ are regular
fibers. The function $\phi_0$ (resp. $\phi_1$) is constant on $C_i$ for $2\leq i \leq k$ (resp. $3\leq i \leq k$), and we may assume that it is vanishing there.
Moreover $\phi_0$ (resp. $\phi_1$) restricted to $C_1$ (resp. $C_0$ and $C_2$) is an isomorphism.
 
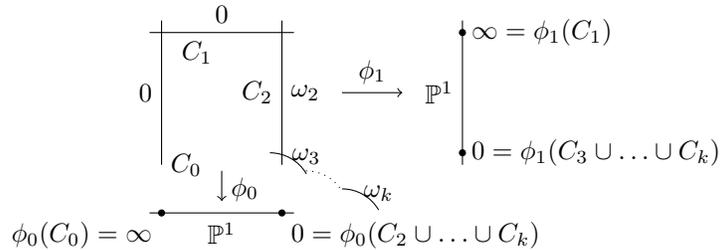
\begin{center}
  \begin{tikzpicture}[scale=0.8]
   \draw (-0.2,0)--(2.2,0) node[above] at (1,0) {$0$} node[below left] at (1,0) {$C_1$};
\draw (0,0.2)--(0,-2.2) node[left] at (0,-1) {$0$} node[right]{$C_0$};
\draw (2,0.2)--(2,-2.2) node[right] at (2,-1) {$\omega_2$} node[left] at (2,-1) {$C_2$};
\draw (1.8,-2) arc (90:30:0.7) node[above] {$\omega_3$};
\draw[dotted] (2.4,-2.3) arc (90:30:0.7);
\draw (3,-2.6) arc (90:30:0.7) node[above] {$\omega_k$};
\draw[->] (3,-1)--(4,-1) node[above] at (3.5,-1){$\phi_1$};
\draw[->] (1,-2.3)--(1,-2.8) node[right] at (1,-2.6){$\phi_0$};
\draw (-0.2,-3)--(2.2,-3) node[below] at (1,-3){$\P^1$};
\draw[fill] (0,-3) circle(0.05) node[below left]{$\phi_0(C_0)=\infty$};
\draw[fill] (2,-3) circle(0.05) node[below right]{$0=\phi_0(C_2\cup\ldots\cup C_k)$};
\draw (5,0.2)--(5,-2.2) node[left] at (5,-1) {$\P^1$};
\draw[fill] (5,0) circle (0.05) node[right] {$\infty = \phi_1(C_1)$};
\draw[fill] (5,-2) circle (0.05) node[right] {$0 = \phi_1(C_3\cup\ldots\cup C_k)$};
  \end{tikzpicture}
 \end{center}
Hence the map $\phi = \phi_0 \times \phi_1: X \rightarrow
\P_x^1\times\P_y^1$ induces isomorphisms $\phi\vert_{C_0}: C_0 \rightarrow \lbrace x=\infty\rbrace$, $\phi\vert_{C_1}: C_1 \rightarrow \lbrace y=\infty\rbrace$
and $\phi\vert_{C_2} \rightarrow \lbrace x = 0 \rbrace$, and moreover $\phi$ contracts the curves $C_3, \ldots, C_k$ onto $(0,0)$. Altogether the
map $\phi$ describes a way how to construct a Gizatullin surface starting with $\C^2$ and blowing up points on $\lbrace x = 0 \rbrace$. The exceptional divisor
of $\phi$ consists of the curves $C_3,\ldots,C_k$ and additional curves (called feathers) $F_1,\ldots, F_n$ that intersect the surface $S$. By Proposition
\ref{fibrations}(b) the number of feathers is precisely $\chi(S)$.

Now, we are able to state some results about rational fibrations on Gizatullin surfaces. Propositions \ref{c-fib}, \ref{cs-fib} and \ref{cs-fib-ds} are
specializations of Proposition 6.6 in \cite{wfhs}. In order to be self-contained we still present complete proofs. Let us start with $\C$-fibrations.

\begin{proposition}[\cite{FKZ}] \label{c-fib}
 Let $f: S \rightarrow \C$ be a $\C$-fibration on a Gizatullin surface $S$. Then there is a standard completion $X$ such that $f$ coincides with the fibration
$\phi_0$ given as above.
\end{proposition}
\begin{proof}
 Let $X$ be a pseudo-minimal SNC-completion of $S$ such that $f$ extends to a regular function $\bar f$. A general fiber of $\bar f$ intersects $D=X\setminus
S = C_0 \cup \ldots \cup C_k$ in precisely one point, therefore one curve in $D$ (say $C_1$) is a section of $\bar f$, and on every other
curve in $D$ the function $\bar f$ is constant. The set $f^{-1}(\infty)\subset D$ is contractible to a rational curve (apply \ref{fibrations}(c) to a the
desingularisation of $X$) which intersects $C_1$ transversally
(since $C_1$ is a section). So by pseudo-minimality $\bar f^{-1}(\infty)$ is already an irreducible curve (say $C_0$) with self-intersection 0. Moreover, by the
absence of further sections, $C_0$ is disjoint from $C_2,\ldots,C_k$. Assume that the
dual graph $\Gamma_X$ is not linear and let $C_i$ be a vertex of degree $\geq 3$. By Proposition \ref{uniquecompletion}(2) all but two branches at $C_i$
are contractible, but by pseudo-minimality the only branch that could be not minimal is the one containing $C_1$. On the
other hand, the branch containing $C_1$ cannot be contractible since it contains also the vertex $C_0$, which has weight 0, and thus is not contractible.
Altogether $\Gamma_X$ is linear and of
the form $\Gamma_X = [[0,n,\omega_2,\ldots,\omega_k]]$ with $n$ arbitrary and $\omega_i \leq -2$, and can be transformed using the modification (\ref{makezero})
into a standard completion such that the fibration $\phi_0$ coincides with $\bar f$.
\end{proof}

\begin{corollary}[\cite{FKZ}] \label{cor-c-fib}
 For a Gizatullin surface $S$ there are as many $\C$-fibrations up to an automorphism as there are standard completions of $S$ up to an isomorphism.
\end{corollary}

Note that there are families of Gizatullin surfaces that have a unique standard completion up to reversion. Check \cite{FKZ} for a 
description of such Gizatullin surfaces. The surfaces called Danielewski surfaces that are introduced in Section 4 are of this kind. Later, Corollary 
\ref{cor-c-fib} will be used e.g. in Proposition \ref{uniquecompletion}.

For $\C^*$-fibrations there are two different cases: The fibration could have either two sections at infinity or one double-section at infinity (i.e. a curve
such that the fibration restricted to this curve is a ramified 2-sheeted covering). First we deal with the
case when there are two sections.

\begin{proposition} \label{cs-fib}
 Let $f: S\rightarrow \C$ be a $\C^*$-fibration on a Gizatullin surface $S$, and let $Y$ be a pseudo-minimal SNC-completion of $S$ such that the
boundary divisor $Y\setminus S$ contains two sections. 

(1) We may choose $Y$ such that the dual graph is of the form 
\begin{center}
\begin{tikzpicture}
\node[left] at (0,0) {$\Gamma_Y = \quad\quad$};
\draw (0,0)--(0.2,0);
\draw[dotted] (0.2,0)--(1,0);
\draw (0.8,0) -- (5.2,0);
\draw[dotted] (5,0)--(5.8,0);
\draw (5.8,0)--(6,0);
\draw[fill] (0,0) circle(0.05) node[above]{$\eta_{-m}$} node[below]{$D_{-m}$};
\draw[fill] (1,0) circle(0.05) node[above]{$\eta_{-2}$} node[below]{$D_{-2}$};
\draw[fill] (2,0) circle(0.05) node[above]{$0$} node[below]{$D_{-1}$};
\draw[fill] (3,0) circle(0.05) node[above]{$0$} node[below]{$D_0$};
\draw[fill] (4,0) circle(0.05) node[above]{$\eta_1$} node[below]{$D_1$};
\draw[fill] (5,0) circle(0.05) node[above]{$\eta_2$} node[below]{$D_2$};
\draw[fill] (6,0) circle(0.05) node[above]{$\eta_n$} node[below]{$D_n$};
\end{tikzpicture}
\end{center}
with $m,n\geq 0$, $\eta_1\leq -1$ and $\eta_i \leq -2$ for $|i|\geq 2$ and additionally $D_0=\bar f^{-1}(\infty)$.

(2) There is a $\omega_1$-semistandard completion $X\supset S$ with $\omega_1 \geq 0$ and
\begin{center}
\begin{tikzpicture}
\node[left] at (0,0) {$\Gamma_X = \quad\quad$};
\draw (0,0)--(2.2,0);
\draw[dotted] (2.2,0)--(2.8,0);
\draw (2.8,0)--(3,0);
\draw[fill] (0,0) circle(0.05) node[above]{$0$} node[below]{$C_0$};
\draw[fill] (1,0) circle(0.05) node[above]{$\omega_1$} node[below]{$C_1$};
\draw[fill] (2,0) circle(0.05) node[above]{$\omega_2$} node[below]{$C_2$};
\draw[fill] (3,0) circle(0.05) node[above]{$\omega_k$} node[below]{$C_k$};
\end{tikzpicture}
\end{center}
such that $Y$ is obtained from $X$ by (i) a sequence of inner (unless $S=\C^2$) blow ups at infinitely near points 
followed by (ii) a modification of type (\ref{movezeros}):
$$ [[0,\omega_1,\omega_2,\ldots,\omega_k]] \stackrel{(i)}{\rightsquigarrow} [[0,0,\eta_{-m},\ldots,\eta_{-2},\eta_1,\eta_2,\ldots,\eta_n]]$$
$$ \stackrel{(ii)}{\rightsquigarrow} [[\eta_{-m},\ldots,\eta_{-2},0,0,\eta_1,\ldots,\eta_n]].$$
\end{proposition}
\begin{proof}
 The proof of (1) works very similarly to the proof above. Again, by Proposition \ref{fibrations}(c), $\bar f^{-1}(\infty)$ is contractible so a curve $C$ with
self-\break intersection equal to 0 and the two sections $C'$ and $C''$ intersect $C$ transverally since they are sections. Moreover, we may assume that 
$C'$ and $C''$
intersect $C$ in two different points. Indeed, otherwise blow up the common intersection point and blow down the strict transform of $C$, and repeat this
procedure until $C'$ and $C''$ intersect $\bar f^{-1}(\infty)$ in two different points. Thus we get an SNC-completion $Y$ with $Y\setminus S =C\cup C'\cup
C''\cup C_1\cup\ldots\cup C_l$ such that $C\cdot C= 0$, $C\cdot C'=1$, $C\cdot C''=1$ and $C$ is disjoint from $C_1\cup\ldots\cup C_l$.
Assume again that the dual graph $\Gamma_Y$ is not linear. Then for a vertex of degree $\geq 3$ all but two branches are contractible, see Proposition
\ref{uniquecompletion}(2). But by pseudo-minimality a
contractible branch must contain one of the curves $C'$ or $C''$. However, then it also
contains the zero vertex corresponding to $C$ and hence it is not contractible. So we have the following picture (note that $D_2$ and $D_{-2}$ may or may not be 
in the same fiber)
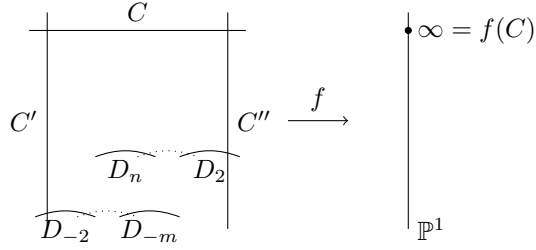
\begin{center}
 \begin{tikzpicture}[scale=0.8]
  \draw (-0.3,3)--(3.3,3) node[above] at (1.5,3) {$C$};
 \draw (0,-0.3)--(0,3.3) node[left] at (0,1.5) {$C'$};
 \draw (3,-0.3)--(3,3.3) node[right] at (3,1.5) {$C''$};
 \draw (-0.2,-0.1) to [out=20,in=160] (0.8,-0.1) node[below] at (0.3,0){$D_{-2}$};
\draw[dotted] (0.5,-0.1) to [out=20,in=160] (1.5,-0.1);
\draw (1.2,-0.1) to [out=20,in=160] (2.2,-0.1) node[below] at (1.7,0){$D_{-m}$};
 \draw (0.8,0.9) to [out=20,in=160] (1.8,0.9) node[below] at (1.3,1){$D_{n}$};
\draw[dotted] (1.5,0.9) to [out=20,in=160] (2.5,0.9);
\draw (2.2,0.9) to [out=20,in=160] (3.2,0.9) node[below] at (2.7,1){$D_{2}$};
 \pgftransformxshift{4cm};
 \draw[->] (0,1.5)--(1,1.5) node[above] at (0.5,1.5){$f$};
 \pgftransformxshift{2cm};
 \draw (0,3.3)--(0,-0.3) node[right] {$\P^1$};
 \draw[fill] (0,3) circle(0.05) node[right] {$\infty = f(C)$};
 \end{tikzpicture}

\end{center}
 and thus $\Gamma_Y = [[\ldots,
\eta_{-2},a,0,b,\eta_2, \ldots]]$. This completion may be transformed by modifications (\ref{movezeros}) into the desired form.

Claim (2) follows from the fact that the graph $\tilde\Gamma_Y=[[\eta_{-m},\ldots,\eta_n]]$ can be contracted to a minimal graph
$\tilde\Gamma=[[\omega_2,\ldots,\omega_k]]$ such that at least one endvertex of $\tilde\Gamma_Y$ does not get contracted. Indeed $\tilde\Gamma_Y$ has at most
one (-1)-vertex. If the right  endvertex $D_n$ is not contracted then move the zeros in $\Gamma_Y$ to the left $[[0,0,\eta_{-m},\ldots,\eta_n]]$, and then
make the contraction by only inner blow downs onto a completion with dual graph $[[0,\omega_1,\omega_2,\ldots,\omega_k]]$.
If the left endvertex $D_{-m}$ is not contracted then repeat the same procedure by moving the zeros to the right $[[\eta_{-m},\ldots,\eta_{n},0,0]]$.
\end{proof}

The $\omega_1$-standard completion from the above proposition can be transformed into a standard completion by modifications (\ref{makezero}) and there are
$\omega_1$ parameters occuring in this process. Therefore we get the following corollary:

\begin{corollary}
Let $S$ be a Gizatullin surface such that each standard completion is determined by its dual graph\footnote{A criterion for this property can be extracted from
\cite{FKZ}. It applies to Danielewski surfaces.}. Then the family of $\C^*$-fibrations having a
pseudo-minimal
SNC-completion with a given dual graph that is obtained as in Proposition \ref{cs-fib} from a $\omega_1$-semistandard completion has at most $\omega_1$
parameters.
\end{corollary}

Let us take a closer look how the fibers of a $\C^*$-fibration $f: S \rightarrow \C$ with two sections at the boundary can look like. For simplicity assume
that the surface $S$ is smooth. Clearly every fiber $f^{-1}(a)$ has precisely one connected component isomorphic to $\C^*$ or to $\C\vee\C$ (two lines
intersecting transversally in one point) namly the one connecting $D_{-m}\cup\ldots\cup D_{-1}$ to
$D_1\cup\ldots\cup D_n$. All other connected components are isomorphic to $\C$, clearly all these $\C$ components are adjecent to a curve
$D_{-m},\ldots,D_{-2},D_2\ldots D_n$. By Proposition \ref{fibrations}(b) the total number of $\C$ and $\C\vee\C$ components is equal to $\chi(S)$. 

The next proposition will clarify the last possibility, namely when there is a double-section at infinity.

\begin{proposition} \label{cs-fib-ds}
 Let $f: S\rightarrow \C$ be a $\C^*$-fibration on a Gizatullin surface $S$, and let $X$ be a pseudo-minimal SNC-completion of $S$ such that $D=X\setminus S$
contains a double-section $C$. Then $\Gamma_X$ is of the form:
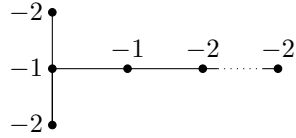
\begin{center}
\begin{tikzpicture}
 \draw (0,0.75)--(0,-0.75)--(0,0)--(2.2,0);
 \draw[dotted] (2.2,0)--(2.8,0);
 \draw (2.8,0)--(3,0);
 \draw[fill] (0,0.75) circle (0.05) node[left]{$-2$};
\draw[fill] (0,-0.75) circle (0.05) node[left]{$-2$};
\draw[fill] (0,0) circle (0.05) node[left]{$-1$};
\draw[fill] (1,0) circle (0.05) node[above]{$-1$};
\draw[fill] (2,0) circle (0.05) node[above]{$-2$};
\draw[fill] (3,0) circle (0.05) node[above]{$-2$};
\end{tikzpicture}
\end{center}
In particular this situation only occurs when the dual graph of a standard completion of $S$ is of the form $$[[0,0,-4]], \quad [[0,0,-3,-3]] \quad \text{or} 
\quad [[0,0,-3,-2,\ldots,-2,-3]].$$
\end{proposition}

\begin{proof}
 Let $C'$ be the double-section. There is a contraction $\pi$ such that the set $\bar f^{-1}(\infty)$ is contractible to a curve $C$ with $C\cdot \pi(C') = 2$
so the curves $C$ and $\pi(C')$ do not intersect
transversally, indeed otherwise they would intersect in two points and the dual graph $\Gamma_X$ would contain a loop. We can see that the
dual graph of $\bar f^{-1}(\infty)$ is $[[-2,-1,-2]]$ and the double-section $C'$ intersects the (-1)-curve transversally. Indeed after two blow ups the
boundary is a SNC-divisor:
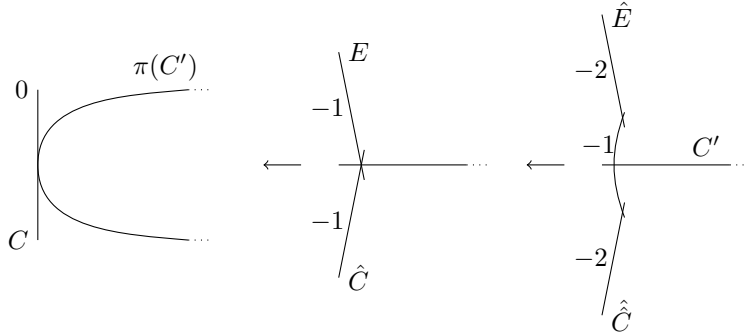
\begin{center}
 \begin{tikzpicture}
  \draw (0,-1)--(0,1) node[left]{$0$} node[left] at (0,-1){$C$};
  \draw (2,1) to  [out=185,in=90] (0,0) to [out=270,in=175] (2,-1) node[above] at (1.7,1){$\pi(C')$};
  \draw[dotted] (2,1) -- (2.3,1);
\draw[dotted] (2,-1) -- (2.3,-1);
  \draw[->] (3.5,0)--(3,0);
\pgftransformxshift{4cm};
 \draw (0,0)--(1.7,0);
 \draw[dotted] (1.7,0)--(2,0);
 \draw (0,-1.5)--(0.34,0.2) node[left] at (0.2,-0.75){$-1$} node[right] at (0,-1.5) {$\hat C$};
\draw (0,1.5)--(0.34,-0.2) node[left] at (0.2,0.75){$-1$} node[right] at (0,1.5) {$E$};
  \draw[->] (3,0)--(2.5,0);
\pgftransformxshift{3.5cm};
 \draw (0,0)--(1.7,0) node[above left]{$C'$};
 \draw[dotted] (1.7,0)--(2,0);
 \draw (0,-2)--(0.3,-0.5) node[left] at (0.2,-1.25){$-2$} node[right] at (0,-2) {$\hat{\hat C}$};
\draw (0,2)--(0.3,0.5) node[left] at (0.2,1.25){$-2$} node[right] at (0,2) {$\hat E$};
 \draw (0.3,0.7) to [out=250,in=110] (0.3,-0.7) node[above left] at (0.3,0) {$-1$};
 \end{tikzpicture}
\end{center}
So $\Gamma_X$ is of the form:
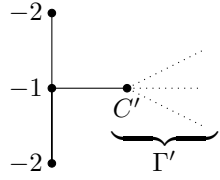
\begin{center}
\begin{tikzpicture}
 \draw (0,1)--(0,-1)--(0,0)--(1,0);
 \draw[dotted] (1,0)--(2,0.5);
 \draw[dotted] (1,0)--(2,-0.5);
 \draw[dotted] (1,0)--(2,0);
 \draw[fill] (0,1) circle (0.05) node[left]{$-2$};
\draw[fill] (0,-1) circle (0.05) node[left]{$-2$};
\draw[fill] (0,0) circle (0.05) node[left]{$-1$};
\draw[fill] (1,0) circle (0.05) node[below]{$C'$};
\node at (1.5,-0.8) {$\stackrel{\underbrace{\quad\quad\quad\quad}}{\Gamma'}$};
\end{tikzpicture}
\end{center}
Since, by Proposition \ref{uniquecompletion}(2), $\Gamma_X$ can be transformed into a linear graph the branch $\Gamma'$ is contractible and by pseudo-minimality
the only (-1)-curve in $\Gamma'$ is $C'$.
This shows that $\Gamma_X$ is of the desired form. After the contraction of $\Gamma$ we get a dual graph of the form $[[-2,n,-2]]$ with $n\geq 0$ and they all
lead to a standard completion as in the claim.
\end{proof}
\begin{remark}
In \cite[Lemmas 4.7+4.8]{kll} the $\C^*$-fibrations on affine toric surfaces were classified using other techniques. Affine toric surfaces are Gizatullin
surfaces and some of them have a completion as in Proposition \ref{cs-fib-ds}. Therefore it is expected that they have a $\C^*$-fibration that, in some sense,
looks essentially different from the other $\C^*$-fibrations. In fact it is possible to see that the twisted $\C^*$-fibrations of affine toric surfaces
correspond exactly to the special case appearing in the end of Lemma 4.8. in \cite{kll}.
\end{remark}

We conclude this section by the classification of $\C^*$-fibrations on $\C^2$. This result is well known: Brunella used it for his classification of complete
vector fields on $\C^2$ in \cite{Brunella}. He cites Suzuki \cite{Suzuki}. Here we give an alternative proof using Lemma \ref{cs-fib}.
\begin{proposition}[\cite{Suzuki}] \label{cs-c2}
 Let $f:\C^2 \rightarrow \C$ be a $\C^*$-fibration. Then (up to an automorphism) $f(x,y)$ is of the form $x^i(x^ly - Q(x))^j$ for $i,j$ relatively prime
numbers, $l\in\N_0$ and a polynomial $Q$ with $\deg(Q)<l$.
\end{proposition}
\begin{proof}
Since the dual graph of a standard completion is $[[0,0]]$, and hence not as the ones in Proposition \ref{cs-fib-ds} there is a pseudo-minimal SNC-completion
$Y$ as in Proposition \ref{cs-fib} with
\begin{center}
\begin{tikzpicture}
\node[left] at (0,0) {$\Gamma_Y = \quad\quad$};
\draw (0,0)--(0.2,0);
\draw[dotted] (0.2,0)--(1,0);
\draw (0.8,0) -- (5.2,0);
\draw[dotted] (5,0)--(5.8,0);
\draw (5.8,0)--(6,0);
\draw[fill] (0,0) circle(0.05) node[above]{$\eta_{-m}$} node[below]{$D_{-m}$};
\draw[fill] (1,0) circle(0.05) node[above]{$\eta_{-2}$} node[below]{$D_{-2}$};
\draw[fill] (2,0) circle(0.05) node[above]{$0$} node[below]{$D_{-1}$};
\draw[fill] (3,0) circle(0.05) node[above]{$0$} node[below]{$D_0$};
\draw[fill] (4,0) circle(0.05) node[above]{$\eta_1$} node[below]{$D_1$};
\draw[fill] (5,0) circle(0.05) node[above]{$\eta_2$} node[below]{$D_2$};
\draw[fill] (6,0) circle(0.05) node[above]{$\eta_n$} node[below]{$D_n$};
\end{tikzpicture}.
\end{center}
Since $\chi(\C^2)=1$ there is precisely one $\C$ or one cross of two lines $\C\vee\C$ inside a fiber (say $f^{-1}(0)$) of $f$. If it is $\C\vee\C$ then by
the Abhyankar-Moh-Suzuki theorem we might assume that the zero set of $f$ is $\lbrace x = 0\rbrace \cup \lbrace y=0 \rbrace$. Hence $f$ is of the form
$x^iy^j$
($l=0$). If
it is a $\C$ component (say $F_1$) then it is attached say to one of the curves $D_2,\ldots,D_n$ and let $F_2$ be the $\C^*$ component of this fiber. By the
absence of other $\C$ components we know that $\bar F_2$ intersects $D_{-m}$ since otherwise $D_{-m}\cup\ldots\cup D_{-2}$ would contain a (-1)-curve. 
By Proposition \ref{cs-fib}(2) we get another completion $X$ of $S$ with
\begin{center}
\begin{tikzpicture}
\node[left] at (0,0) {$\Gamma_X = \quad\quad$};
\draw (0,0)--(1,0);
\draw[fill] (0,0) circle(0.05) node[above]{$0$} node[below]{$C_0$};
\draw[fill] (1,0) circle(0.05) node[above]{$\omega_1$} node[below]{$C_1$};
\end{tikzpicture}.
\end{center}
It is obtained from $Y$ such that $\bar F_1$ is disjoint from $C_0$ and $D_{-m}$ maps isomorphically onto $C_0$. Thus $\bar F_2$ still intersects $C_0$
transversally in one point. We continue by blowing up the point $C_0\cap \bar F_2$ and blowing down the strict transform of $C_0$, which is a modification of
type (\ref{makezero}). Repeating this $\omega_1$-times we will end up with a completion isomorhic to $\P^1\times\P^1$ such that $\overline F_2$ intersects
$\lbrace x = \infty \rbrace$ transversally in one point, and $\overline F_1\cap \lbrace
x=\infty \rbrace=\emptyset$. Hence in these coordinates we may assume that $F_1 = \lbrace x=0\rbrace$ and $F_2 = \lbrace y=R(x)\rbrace$ is a graph for a 
rational
function $R$ with a pole at $0$. Let us write $R$ as $R(x)=Q(x)x^{-l} + P(x)$ for some polynomials $Q$ and $P$. Then after a coordinate change given by
$(x,y)\mapsto(x,y+P(x))$ we get that $F_2 =\lbrace x^ly = Q(x)\rbrace$, and thus the
claim follows.
\end{proof}

\section{$\C$- and $\C^*$-fibrations on smooth Danielewski surfaces}
Danielewski surfaces form a subfamily of Gizatullin surfaces. They have an explicit description as a hypersurface in $\C^3$ and the classification of $\C$-
and $\C^*$- fibrations can be done very explicit. In most cases the classification looks exactly the same as the classification of $\C$- and $\C^*$- fibrations
on $\C^2$. It is a direct consequence of the famous Abhyankar-Moh-Suzuki theorem that all $\C$-fibrations are up to an automorphism given by the projection to 
the $x$-coordinate. Actually, this classification has already been found by Gutwirth \cite{Gu}. Proposition \ref{cs-c2} is the description of 
$\C^*$-polynomials on $\C^2$. 
\begin{definition}
 A smooth affine surface $S$ in called \textit{Danielewski surface}\footnote{In the literature surfaces given by $\lbrace x^n y=p(z)\rbrace$ are often also 
called Danielewski surfaces. In this text we only consider the case $n=1$.} if there is an SNC-completion $X$ such that $\Gamma_X = [[0,0,-k]]$ for $k\geq 2$. 
Danielewski
surfaces can also be seen as surfaces in $\C^3$ given by the equation $\lbrace xy=p(z)\rbrace$ for a polynomial $p$ of degree $k$ with simple zeros. 
\end{definition}
Let $p$ be a polynomial of degree $k$ with simple zeros. Given the surface $S=\lbrace xy = p(z) \rbrace \subset
\C^3$ it is easy to construct a standard completion. The projection $\pi(x,y,z) = (x,z)$ is a birational map from $S$ to $\C^2$, it is an isomorphism on the
open sets $\lbrace x \neq 0 \rbrace$ and it contracts the lines $\lbrace x = 0, z=z_i \rbrace$ onto the points $(0,z_i)$ where the numbers $z_i$ are the zeros
of $p$. So $S$ is isomorphic to an open set in $\C^2$ blown up in the points $(0,z_i)$ and therefore $\P^1\times\P^1$ blown up in these points is a completion
$X_0$ of $S$. The curve 
$$D_0=X_0\setminus S = \widehat{\lbrace x = \infty\rbrace}\cup\widehat{\lbrace z=\infty\rbrace}\cup\widehat{\lbrace x=0\rbrace}$$ 
is the boundary with dual graph $\Gamma_{X_0}
=[[0,0,-k]]$ (where $\hat C$ denotes the strict transform of a curve $C$). Moreover, the projection to the $x$ (resp. $z$) coordinate corresponds to the map
$\phi_0$ (resp. $\phi_1$) constructed in the previous section
and therefore the map $\pi$ corresponds to the map $\phi$.

On the other hand, given any standard completion of a Danielewski surface $S$, its corresponding map $\phi$ will describe a way to embed $S$ into $\C^3$. 
Indeed, $S$ is given in $\C^3$ by the equation $xy=p(z)$, when the polynomial $p$ is defined such that its zeros are the indeterminacy points of $\phi^{-1}$).

We begin with the description of $\C$-fibrations on $S$:
\begin{proposition}[\cite{ML,Daigle,FKZ,BD}]\label{uniquecompletion}
 Let $f: S = \lbrace xy=p(z)\rbrace \rightarrow \C$ be a $\C$-fibration. 

(1) Up to automorphism of $S$ the fibration $f$ is given by the projection $f(x,y,z)=x$.

(2) Any standard completion of $S$ is isomorphic to the standard completion $X_0$ constructed above.
\end{proposition}
\begin{proof}
By Corollary \ref{cor-c-fib} (1) is equivalent to (2). There are several proofs, e.g. (1) is proven in \cite{ML}
and (2) is proven in \cite{FKZ}.
\end{proof}

\subsection{$\C^*$-fibrations with two sections at the boundary}
The description of $\C^*$-fibration with two sections at the boundary is very much related with the description of $\C^*$-fibrations on $\C^2$. We
will prove the following proposition:
\begin{proposition}
 Let $f: S=\lbrace xy=p(z) \rbrace \rightarrow \C$ be a $\C^*$-fibration with two sections at the boundary. Then $f$ is up to isomorphism of $S$ of the form
$z$ or $x^i(x^l(z+a) + Q(x))^j$ for $i,j$ relatively prime, $l\in\N_0$, $\deg(Q)<l$ and $a\in \C$.
\end{proposition}
\begin{proof}
 Let $X\supset S$ be the semistandard completion from Proposition \ref{cs-fib} with
\begin{center}
\begin{tikzpicture}
\node[left] at (0,0) {$\Gamma_X = \quad\quad$};
\draw (0,0)--(2,0);
\draw[fill] (0,0) circle(0.05) node[above]{$0$} node[below]{$C_0$};
\draw[fill] (1,0) circle(0.05) node[above]{$\omega_1$} node[below]{$C_1$};
\draw[fill] (2,0) circle(0.05) node[above]{$-k$} node[below]{$C_2$};
\end{tikzpicture}
\end{center}
that is obtained (starting by moving the zeros to the left followed by inner blow downs) from a pseudo-minimal SNC-completion $Y$ with
\begin{center}
\begin{tikzpicture}
\node[left] at (0,0) {$\Gamma_Y = \quad\quad$};
\draw (0,0)--(0.2,0);
\draw[dotted] (0.2,0)--(1,0);
\draw (0.8,0) -- (5.2,0);
\draw[dotted] (5,0)--(5.8,0);
\draw (5.8,0)--(6,0);
\draw[fill] (0,0) circle(0.05) node[above]{$\eta_{-m}$} node[below]{$D_{-m}$};
\draw[fill] (1,0) circle(0.05) node[above]{$\eta_{-2}$} node[below]{$D_{-2}$};
\draw[fill] (2,0) circle(0.05) node[above]{$0$} node[below]{$D_{-1}$};
\draw[fill] (3,0) circle(0.05) node[above]{$0$} node[below]{$D_0$};
\draw[fill] (4,0) circle(0.05) node[above]{$\eta_1$} node[below]{$D_1$};
\draw[fill] (5,0) circle(0.05) node[above]{$\eta_2$} node[below]{$D_2$};
\draw[fill] (6,0) circle(0.05) node[above]{$\eta_n$} node[below]{$D_n$};
\end{tikzpicture}
\end{center}
with $\eta_i\leq -2$ for $|i|\geq 2$ and $\eta_{-1}\leq -1$. Since $[[\eta_{-m},\ldots,\eta_{-2},\eta_1,\ldots,\eta_n]]$ is contractible to $[[-k]]$ such
that the right endvertex is not contracted we have $\eta_1=-1$ and thus $n\geq 2$ unless $X=Y$. We may extend $f: S\rightarrow \C$ to a rational function $\bar
f: X \dashrightarrow \P^1$. If $X=Y$ then $f(x,y,z)=z$ up to isomorphism. Indeed, by
Proposition \ref{uniquecompletion} the completion $X=Y$ is isomorphic to  $X_0$, and $\bar f:X\cong X_0\rightarrow \P^1$ coincides with $\phi_1$ which is the
projection to the
$z$-coordinate. If $X\neq Y$ then by construction $\bar f$ is constant and non-polar on $C_2\setminus C_1$ (assume that $f$ vanishes on $C_2\setminus C_1$). 
Indeed, $C_2$ is the strict
transform of $D_{n}$ which sits inside a fiber (since $n\geq2$). The same holds true if we pass by modifications (\ref{makezero}) to a standard completion $X'$.
Hence the
pushforward $\phi_* \bar f$ by the morphism $\phi: X'\cong X_0 \rightarrow\P^1\times\P^1$ restricted to $\C^2$ is a regular function $g:=\phi_*\bar
f \vert_{\C^2}: \C^2 \rightarrow \C$. In particular, $g$ is a polynomial function on $\C^2$ with general fibers isomorphic to $\C^*$ and $\lbrace x = 0 \rbrace
\subset g^{-1}(0)$. By Proposition \ref{cs-c2} the function $g$ and hence its pullback $f$ is, for some automorphism $(s,t)$ of $\C^2$, of the form
$s(x,z)^i(s(x,z)^lt(x,z) - Q(s(x,z)))^j$
with $i,j,l,Q$ as desired. Clearly we have that (if $l=0$ then maybe after exchanging $s$ and $t$) the zero set of
$s(x,z)$ coincides with $\lbrace x=0\rbrace$. Hence the automorphism is of the form
$s(x,z)=ax$ and $t(x,z)=by+r(x)$, and after rescaling $f$ we may assume $a=b=1$. Since automorphisms of $\C^2$ of the form $(x,z)\mapsto(x,z+xr'(x))$ extend to
the surface $S$ we may even assume that $s(x,z)=x$ and $t(x,z)=z+a$ for some $a\in\C$ and the claim follows.
\end{proof}

\subsection{$\C^*$-fibrations with one double-section at the boundary}
By Proposition \ref{cs-fib-ds} the case of a $\C^*$-fibration with a double-section at the boundary on a Danielewski surface only occurs when the polynomial
$p$ is of degree $4$. It will be more convenient to
allow completions where the components of the boundary do not necessarily intersect transversally. 

\begin{lemma}\label{lem-cs-dan}
 Let $X$ be a non-SNC-completion of $S=\lbrace xy=a(z-z_1)(z-z_2)(z-z_3)(z-z_4)\rbrace$ such that $X\setminus S = C_0 \cup C_1$ with $C_0\cdot C_1 =2$,
$C_0\cdot C_0=0$ and $C_1\cdot C_1 = 1$. Then:

(1) $X$ can be identified with $\P^2$ blown up in $[z_i^2:z_i:1]$ for $1\leq i\leq 4$ such that
\[
 C_0=\widehat{\lbrace w=0\rbrace} \quad \mathrm{and} \quad C_1=\widehat{\lbrace uw = v^2\rbrace}.
\]

(2) There is a unique (up to affine transformation of $\P^1\setminus\lbrace\infty\rbrace$) rational function $h:X\rightarrow \P^1$ such that $C_0$ is a
double-section and $C_1 =
h^{-1}(\infty)$.
The pushforward $\tilde h$ of $h$ to $\P^2$ is given by 
\[
 \tilde h([u:v:w]) = \frac{(u-(z_1 + z_2)v+ z_1z_2w)(u-(z_3 + z_4)v+ z_3z_4w)}{uw-v^2}.
\]
Moreover, $h$ has at least three fibers which are not isomorphic to $\C^*$.
\end{lemma}
\begin{proof}
 The completion $X$ may be transformed into a standard completion by the following modifications:
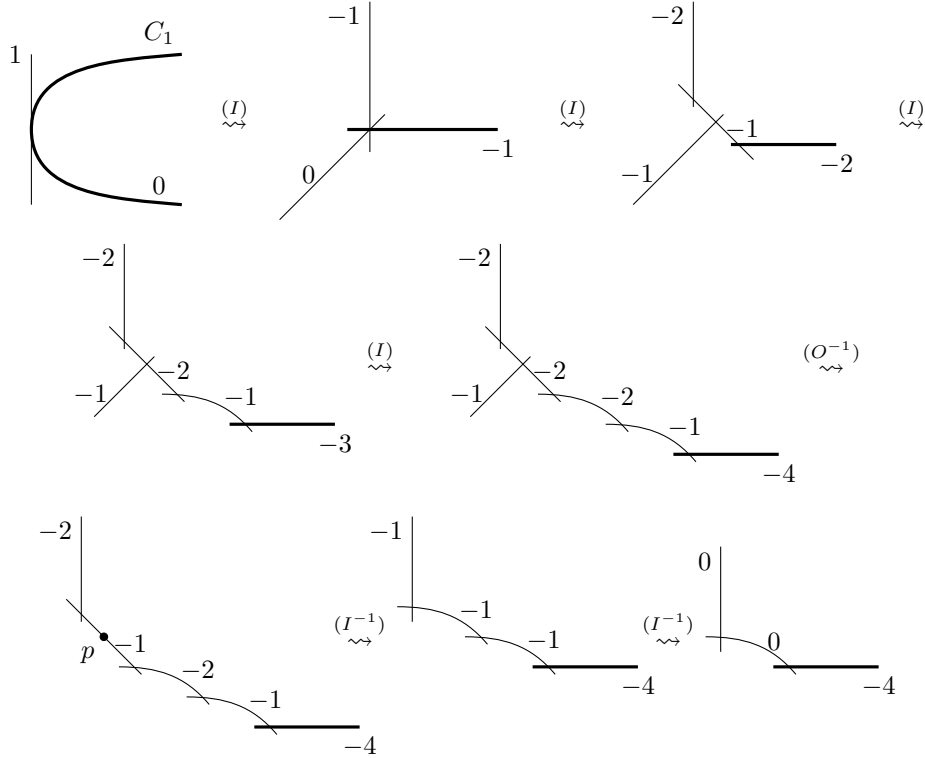
\begin{center}
 \begin{tikzpicture}
    \draw (0,-1)--(0,1) node[left]{$1$};
  \draw[very thick] (2,1) to  [out=185,in=90] (0,0) to [out=270,in=175] (2,-1) node[above] at (1.7,1){$C_1$} node[above] at (1.7,-1){$0$};
\node at (2.7,0.2) {$\stackrel{(I)}{\rightsquigarrow}$};
\pgftransformxshift{4.5cm};
 \draw[very thick] (-0.3,0)--(1.7,0) node[below] {$-1$};
 \draw (-1.2,-1.2)--(0.2,0.2) node[left] at (-0.6,-0.6){$0$};
\draw (0,1.7)--(0,-0.3) node[left] at (0,1.5){$-1$};
\node at (2.7,0.2) {$\stackrel{(I)}{\rightsquigarrow}$};
\pgftransformxshift{4.5cm};
 \draw[very thick] (0.3,-0.2)--(1.7,-0.2) node[below] {$-2$};
 \draw (-1,-1)--(0.2,0.2) node[left] at (-0.6,-0.6){$-1$};
\draw (-0.2,1.7)--(-0.2,0.3) node[left] at (-0.2,1.5){$-2$};
\draw (-0.4,0.6)--(0.6,-0.4) node[right] at (0.1,0){$-1$};
\node at (2.7,0.2) {$\stackrel{(I)}{\rightsquigarrow}$};
 \end{tikzpicture}
\end{center}
\begin{center}
 \begin{tikzpicture}
   \draw[very thick] (1.2,-0.7)--(2.6,-0.7) node[below] {$-3$};
 \draw (-0.6,-0.6)--(0.2,0.2) node[left] at (-0.3,-0.3){$-1$};
\draw (-0.2,1.7)--(-0.2,0.3) node[left] at (-0.2,1.5){$-2$};
\draw (-0.4,0.6)--(0.6,-0.4) node[right] at (0.1,0){$-2$};
\draw (0.3,-0.3) to [out=0,in=130] (1.5,-0.8) node[above right] at (1,-0.6){$-1$};
    \node at (3.2,0.2) {$\stackrel{(I)}{\rightsquigarrow}$};
\pgftransformxshift{5cm};
   \draw[very thick] (2.1,-1.1)--(3.5,-1.1) node[below] {$-4$};
 \draw (-0.6,-0.6)--(0.2,0.2) node[left] at (-0.3,-0.3){$-1$};
\draw (-0.2,1.7)--(-0.2,0.3) node[left] at (-0.2,1.5){$-2$};
\draw (-0.4,0.6)--(0.6,-0.4) node[right] at (0.1,0){$-2$};
\draw (0.3,-0.3) to [out=0,in=130] (1.5,-0.8) node[above right] at (1,-0.6){$-2$};
\draw (1.2,-0.7) to [out=0,in=130] (2.4,-1.2) node[above right] at (1.9,-1){$-1$};
\node at (4.2,0.2) {$\stackrel{(O^{-1})}{\rightsquigarrow}$};
 \end{tikzpicture}
\end{center}
\begin{center}
 \begin{tikzpicture}
  \draw[very thick] (2.1,-1.1)--(3.5,-1.1) node[below] {$-4$};
\draw[fill] (0.1,0.1) circle (0.05) node[below left] {$p$};
\draw (-0.2,1.7)--(-0.2,0.3) node[left] at (-0.2,1.5){$-2$};
\draw (-0.4,0.6)--(0.6,-0.4) node[right] at (0.1,0){$-1$};
\draw (0.3,-0.3) to [out=0,in=130] (1.5,-0.8) node[above right] at (1,-0.6){$-2$};
\draw (1.2,-0.7) to [out=0,in=130] (2.4,-1.2) node[above right] at (1.9,-1){$-1$};
\node at (3.5,0.2) {$\stackrel{(I^{-1})}{\rightsquigarrow}$};
\pgftransformxshift{3.7cm};
  \draw[very thick] (2.1,-0.3)--(3.5,-0.3) node[below] {$-4$};
\draw (0.5,1.7)--(0.5,0.3) node[left] at (0.5,1.5){$-1$};
\draw (0.3,0.5) to [out=0,in=130] (1.5,0) node[above right] at (1,0.2){$-1$};
\draw (1.2,0.1) to [out=0,in=130] (2.4,-0.4) node[above right] at (1.9,-0.2){$-1$};
\node at (3.9,0.2) {$\stackrel{(I^{-1})}{\rightsquigarrow}$};
\pgftransformxshift{3.2cm};
  \draw[very thick] (2.1,-0.3)--(3.5,-0.3) node[below] {$-4$};
\draw (1.4,1.3)--(1.4,-0.1) node[left] at (1.4,1.1){$0$};
\draw (1.2,0.1) to [out=0,in=130] (2.4,-0.4) node[above right] at (1.9,-0.2){$0$};
 \end{tikzpicture}
\end{center}
A calculation shows that the birational map $\P^1\times\P^1 \dashrightarrow \P^2$ given by $(x,z) \mapsto [u(x,z):v(x,z):w(x,z)]=[x+az^2:z:1]$ induces precisely
the inverse of this modification on the boundary (where $a$ corresponds to the point $p$). Thus $X$ can be identified with $\P^2$ blown up in $[az_i^2:z_i:1]$ 
for
$1\leq i\leq 4$ (indeed, the standard completion was isomorphic to $\P^1\times\P^1$ blown up in $(0,z_i)$ for $1\leq i\leq 4$). After the isomorphism
$[u:v:w]\mapsto
[a^{-1}u:v:w]$ the completion $X$ is as desired. For Claim (2) we observe that $\tilde h$ is of
degree 2, indeed the general fiber meets $\lbrace w=0\rbrace$ twice. Moreover, every fiber meets $\lbrace uw = v^2\rbrace$ precisely in the points
$[z_i^2:z_i:1]$ for $1\leq i\leq 4$. This holds since these points are indeterminacy points of $\tilde h$ because $C_1$ is an entire fiber of $h$. The 
space of 
curves of
degree 2 in $\P^2$ is isomorphic to $\P^5$ hence the space of curves of degree 2 passing through four points is isomorphic to $\P^1$. It coincides with the 
levels
of $\tilde h$. So, $\tilde h$ is (up to affine transformation of $\P^1\setminus\lbrace\infty\rbrace$) of the form $(uw-v^2)^{-1}g(u,v,w)$, where $g$ is any
homogeneous polynomial of degree 2 such that its zero set meets
$\lbrace uw = v^2\rbrace$ in the four requested points. We may choose the product of two linear functions each connecting two of the points linearly. Clearly 
$h$
has at least three fibers not isomorphic to $\C^*$ since there are three possibilities to choose two lines through these four points. 
\end{proof}

\begin{proposition}\label{prop-cs-ds}
 Let $f: S=\lbrace xy=p(z)\rbrace \rightarrow \C$ be a $\C^*$-fibration with double-section at the boundary. Then $\deg p = 4$ and $f$ is given up to
automorphism of $S$ by $f(x,y,z) = ax + y + \frac{1}{6} p''(z)$ where $a$ is the leading coefficient of $p$. Additionally, the fibration $f$ has at least three 
fibers not
isomorphic to $\C^*$.
\end{proposition}
\begin{proof}
 By Proposition \ref{cs-fib-ds} the polynomial $p$ has degree 4 (say $p(z)=a(z-z_1)(z-z_2)(z-z_3)(z-z_4)$), and moreover there is a pseudo-minimal 
SNC-completion
with dual graph of the boundary 
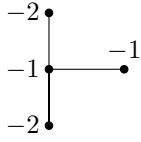
\begin{center}
\begin{tikzpicture}
 \draw (0,0.75)--(0,-0.75)--(0,0)--(1,0);
 \draw[fill] (0,0.75) circle (0.05) node[left]{$-2$};
\draw[fill] (0,-0.75) circle (0.05) node[left]{$-2$};
\draw[fill] (0,0) circle (0.05) node[left]{$-1$};
\draw[fill] (1,0) circle (0.05) node[above]{$-1$};
\end{tikzpicture}.
\end{center}
This completion can be transformed by two blow downs into a completion $X$ as in Lemma \ref{lem-cs-dan}, which is then by (1) isomorphic to $\P^2$ blown up in 4
points. The birational map $\P^1\times\P^1 \dashrightarrow \P^2$ given by $(x,z) \mapsto [u(x,z):v(x,z):w(x,z)]=[x+z^2:z:1]$ induces a birational map from the
standard completion $X_0$ to the completion $X$. By (2) of Lemma \ref{lem-cs-dan} the fibration $f$ is given by
\begin{eqnarray*}
&\displaystyle{\frac{(u-(z_1 + z_2)v+ z_1z_2w)(u-(z_3 + z_4)v+ z_3z_4w)}{uw-v^2}}  &= \\
&\displaystyle{\frac{(x+z^2 -(z_1 + z_2)z + z_1z_2)(x+z^2 -(z_3 + z_4)z + z_3z_4)}{x+z^2 -z^2}}  &= \\
&\displaystyle\frac{1}{x}\left[\begin{array}{c} x^2 + x(2z^2-(z_1+z_2+z_3+z_4)z+ z_1z_2 + z_3z_4) \\ + (z-z_1)(z-z_2)(z-z_3)(z-z_4) \end{array}\right]  &= \\
&x+ \displaystyle\frac{p(z)}{ax} + 2z^2 - (z_1+z_2+z_3+z_4)z+ z_1z_2 + z_3z_4  &= \\
&x +\displaystyle\frac{y}{a} + 2z^2 - (z_1+z_2+z_3+z_4)z+ z_1z_2 + z_3z_4&
\end{eqnarray*}
and hence $f$ is (after multiplying with $a$ and adding a constant) of the desired form.
\end{proof}

\section{Proof of the Main Theorem}

Let $p$ be a polynomial with simple zeros, and let $$\nu=\nu_x(x,y,z)\frac{\partial}{\partial x} + \nu_y(x,y,z)\frac{\partial}{\partial y}+
\nu_z(x,y,z)\frac{\partial}{\partial z}$$ be a complete algebraic vector field on the Danielewski surface $S=\lbrace xy=p(z)\rbrace$ extended regularly to 
$\C^3$. Then, as mentioned in the
Introduction, by \cite[Theorem 1.3]{wfhs} the vector field $\nu$
preserves a $\C$- or $\C^*$-fibration $f: S\rightarrow \C$. These fibrations are described in the previous section. Hence it is possible to
give the precise form of $\nu$ using exactly the same arguments as in the planar case (see Proposition 2 in \cite{Brunella}). Let us establish first two 
lemmas.

\begin{lemma} \label{lem-tangential}
 Assume that $\nu$ is tangent to $\lbrace x = 0\rbrace$. Then $\nu$ projects to a complete vector field $$\nu_x(x,\frac{p(z)}{x},z)\frac{\partial}{\partial x} +
\nu_z(x,\frac{p(z)}{x},z)\frac{\partial}{\partial z}$$ on $\C^*_x \times \C_z$, $\nu_x$ and $\nu_z$ are divisible by $x$ and $\nu$ is of the form $$\nu =
\frac{\nu_x(x,y,z)}{x}\mathrm{HF} + \frac{\nu_z(x,y,z)}{x}\mathrm{SF}^x.$$
\end{lemma}
\begin{proof}
 Regarding $\nu$ as a derivation, clearly $\nu_x=\nu(x)$ vanishes on $\lbrace x=0 \rbrace$. Therefore we have $\nu(z)p'(z)=\nu(p(z))=\nu(xy)=x\nu(y)+y\nu(x)=0$ 
for $x=0$.
Hence $\nu_z=\nu(z)$ vanishes also on  $\lbrace x=0 \rbrace$. This means that both $\nu_x$ and $\nu_z$ are divisible by $x$. Thus we get the 
explicit form of $\nu_y=\nu(y)=(p'(z)\nu_z-y\nu_x)/x$, and also the explicit form of $\nu$ as in the claim.
\end{proof}

\begin{lemma}[\cite{Brunella}]\label{lem-trivialization}
(1) Let $D_\alpha\times C_t$ be a (holomophic) trivialization of a neighborhood of a general fiber $C$ of $f$. Then the pullback of $\nu$ to this neighborhood
is of the form
$$ \tilde\nu = F(\alpha)\frac{\partial}{\partial \alpha} + \left( G(\alpha)t + H(\alpha)\right) \frac{\partial}{\partial t}$$ for holomorphic functions $F$, $G$
and $H$. If $C \cong \C^*$ then $H=0$.

(2) If $\nu = \nu_1 + \nu_2$, where $\nu_2$ is complete and tangent to the fibers of $f$. Then $\nu_1$ is complete.
\end{lemma} 
\begin{proof} Since the local flow of $\tilde\nu$ sends vertical fibers to vertical fibers the first summand of $\tilde\nu$ is independent of $t$. By the 
Riemann
removable singularities theorem the local flow maps of $\tilde\nu$ extends to maps $\lbrace\alpha\rbrace\times\bar C\rightarrow
\lbrace\alpha'\rbrace\times\bar C$. Hence $\tilde\nu$ extends to $D\times\bar C$ such
that $\tilde\nu$ is tangential to $D\times\partial C$. Thus the second summand is of the desired form. The second claim follows from the fact
that also $\nu_1$ extends to $\bar C$ such that it is tangential to the sections at infinity.
\end{proof}

These two lemmas directly imply the next proposition concerning the case of $\C$-fibrations.
\begin{proposition}
 If $f(x,y,z)=x$ then $$\nu = c\mathrm{HF} + \left( A(x)z + B(x) \right)\mathrm{SF}^x$$ for some $c\in\C$ and $A,B\in\C[x]$.
\end{proposition}
\begin{proof}
 Since $\lbrace x = 0\rbrace$ is a singular fiber, $\nu$ is tangential to it. Lemma \ref{lem-tangential} shows that it is sufficient to look at the
projection and restriction of $\nu$ to $\C^*_x\times\C_z$. The latter is obviously a trivialization of a neighborhood of a fiber. Hence Lemma
\ref{lem-trivialization}(1)
shows that $\nu$ is of the form $F(x)\partial/\partial x + \left( G(x)z + H(x)\right)\partial/\partial z$ on $\C^*_x\times\C_z$. By Lemma \ref{lem-tangential}
the functions $F,G,H$ are divisible by $x$. By the completeness of $\nu$ we have $F(x)=cx$ for some $c$, which leads to the desired form.
\end{proof}

The $\C^*$-case with two sections at the boundary works similarly. The only new difficulty is to trivialize a neighborhood of a fiber.
\begin{proposition}
 If $f(x,y,z)= x^m(x^l(z+a)+Q(x))^n$ for coprime numbers $m,n\in\N$, $l\in\N_0$, $a\in\C$ and $\deg(Q) < l$ then 
\begin{eqnarray*} \nu &=& c\left( \frac{z+a}{x} + \frac{Q(x)}{x^{l+1}}\right)\mathrm{SF}^x +
A(x^m(x^l(z+a)+Q(x))^n) \\ &&\cdot\left[n\mathrm{HF}-\left(\frac{(m+nl)(z+a)}{x} + \frac{mQ(x)+nxQ'(x)}{x^{l+1}}\right)\mathrm{SF}^x\right]\end{eqnarray*}
for some $c\in\C$ and $A\in\C[t]$ satisfying $A(0)=c/(m+nl)$ and $A(x^m(x^l(z+a)+Q(x))^n)(mQ(x)+nxQ'(x))-cQ(x) \in x^{l+1}\cdot\C[S]$.
\end{proposition}
\begin{proof}
 Again $\nu$ is tangential to $\lbrace x=0\rbrace$, so we work on $\C^*_x\times\C_z$ as in Lemma \ref{lem-tangential}. Pick $0\neq\alpha_0\in\C$, and let
$D=\lbrace\vert \alpha-\alpha_0\vert<\varepsilon\rbrace$ be a small ball around $\alpha_0$. Then the map
\begin{eqnarray*}
 D\times\C^* &\rightarrow&\C^*_x\times\C_z\\
(\alpha,t)&\mapsto&\left(t^n,\frac{e^\alpha t^{-m}-Q(t^n)}{t^{nl}} - a\right)
\end{eqnarray*}
gives a trivialization of a neighborhood of the fiber $f^{-1}(e^{n\alpha_0})$. Using this map yields:
\begin{eqnarray*}
 \frac{\partial}{\partial \alpha}& \mapsto & \nu_1 := \left(z+a+\frac{Q(x)}{x^l}\right)\frac{\partial}{\partial z},\\
 t\frac{\partial}{\partial t} & \mapsto& \nu_2:=nx\frac{\partial}{\partial x} - \left((m+nl)(z+a) + \frac{mQ(x)+nxQ'(x)}{x^l}\right)\frac{\partial}{\partial
z}.
\end{eqnarray*}
Lemma \ref{lem-trivialization}(1) shows that $\nu$ is given on $\C^*_x\times\C_z$ by $F(\alpha)\nu_1 + G(\alpha)\nu_2$ for $\alpha= x^m(x^l(z+a)+Q(x))^n$. We
know that $G(\alpha)\nu_2$ is complete on $\C^*_x\times\C_z$ since it is tangent along the fibers of $f$ and its restriction to any fiber is complete. 
Thus by Lemma
\ref{lem-trivialization}(2) also $F(\alpha)\nu_1$ is complete on $\C^*_x\times\C_z$. This shows that $F(\alpha)$ is constant. 
Letting $A=G$ yields that $\nu$ is as desired on $\C^*_x\times\C_z$. Lemma \ref{lem-tangential} provides a lift of $\nu$ to the vector field on $S$ as in the 
claim. In
order to be non-polar on $\lbrace x = 0
\rbrace$ we need the
additional condition on $A$, which is equivalent to the fact that $\nu_z$ is divisible by $x$.
\end{proof}

\begin{proposition}
 If $p(z)=a\cdot(z^4 + bz^3 + cz^2 + dz + e)$ and $f(x,y,z)= ax + y + \frac{1}{6}p''(z)$ then $$\nu = A\left(ax + y +
\frac{1}{6}p''(z)\right)\left(-\frac{1}{6}p'''(z)\mathrm{HF} + a\mathrm{SF}^x -\mathrm{SF}^y\right)$$ for some $A\in\C[t]$.
\end{proposition}

\begin{proof}
By Proposition \ref{prop-cs-ds} we know that $f$ has more than one fiber not isomorphic to $\C^*$. Thus $\nu$ acts on the base $\C$ with more that one fixed
point. By hyperbolicity $\nu$ is tangential to the fibers of $f$. Hence $\nu$ restricted to a general fiber is proportional to
$t\partial/\partial t$. We need to parametrize a general fiber $C^\alpha=\lbrace ax + y + 2az^2 + abz + a\alpha = 0 \rbrace$, $\alpha\in\C$. Let us define 
$\xi,\chi,\kappa\in\C$ such that
\[
 \xi^2=\alpha+\frac{b^2}{2}-c,\quad \chi=\frac{\alpha b-2d}{4\xi^2},\quad\kappa=e-\frac{\alpha^2}{4}+\xi^2\chi^2.
\]
The map $C^\alpha\rightarrow\C^*$ defined by 
\[
 (x,y,z) \mapsto t:= x + z^2 + \frac{b}{2}z + \frac{\alpha}{2} + \xi(z+\chi) = \frac{ax-y}{2a} + \xi(z+\chi)
\]
is an isomorphism. Indeed, after multiplying with $x/a$ and replacing $xy$ by $p(z)$ the equation defining $C^\alpha$ becomes
\begin{eqnarray*}
 &x^2 + z^4 + bz^3 + cz^2+dz+e +(2z^2 +bz + \alpha)x & =  \\
&\displaystyle{\left( x +z^2+\frac{b}{2}z + \frac{\alpha}{2}\right)^2 + \left(c-\frac{b^2}{4}-\alpha\right)z^2+\left(d-\frac{\alpha b}{2}\right)z + e
-\frac{\alpha^2}{4} }& = \\
&\displaystyle{\left( x +z^2+\frac{b}{2}z + \frac{\alpha}{2}\right)^2-(\xi(z+\chi))^2 + \kappa }& = \\
&t(t-2\xi(z+\chi)) + \kappa&.
\end{eqnarray*}
Thus $t$ can be seen as a variable of $\C^*$. Moreover, we can see that the vector field $\nu_0 = -\frac{1}{6}p'''(z)\mathrm{HF} + a\mathrm{SF}^x
-\mathrm{SF}^y$ is tangent to the fibers and restricts to the vector field $2a\xi t \partial/\partial t$ on $C^\alpha\cong\C^*$. Indeed, $\nu_0$ acts on $t$ by
multiplication with $2a\xi$:

\vspace{0.4cm}$\begin{aligned}
  \nu_0(t)\quad\quad &= \quad\nu_0\left(\displaystyle{\frac{ax-y}{2a} + \xi(z+\chi)}\right)\\
&= \quad 2a\xi\left(\displaystyle{-\frac{p'''(z)}{6}\cdot\frac{ax+y}{4a^2\xi}-\frac{p'(z)}{2a\xi} + \frac{ax-y}{2a}}\right) \quad = \\
& \hspace{-1.4cm} 
2a\xi\left(\displaystyle-\frac{1}{\xi}\left((4z+b)\frac{-(2z^2+bz+\alpha)}{4}+2z^3+\frac{3}{2}bz^2+cz+\frac{d}{2}\right)+\frac{ax-y}{2a}\right)\\
&=\quad2a\xi\left(\displaystyle-\frac{1}{\xi}\left(\left(-\frac{b^2}{4}-\alpha + c\right)z-\frac{\alpha b}{4}+\frac{d}{2}\right) +\frac{ax-y}{2a}\right)\\
&= \quad 2a\xi\left(\displaystyle\xi(z+\chi) + \frac{ax-y}{2a}\right)\\
&= \quad 2a\xi t
\end{aligned}$\vspace{0.2cm} 

Overall on every fiber of $f$ the vector field $\nu$ is a multiple of $\nu_0$. Thus the proposition is proven.
\end{proof}

\end{document}